\documentclass[leqno,12pt]{amsart}

\usepackage[top=1.5in,bottom=1.3in,left=1.1in,right=1.1in,marginparwidth=1.5cm]{geometry}
\usepackage{times}
\usepackage{amssymb}
\usepackage{microtype}
\usepackage{comment}
\usepackage[english]{babel}
\usepackage[colorlinks=true,citecolor=blue,urlcolor=blue,linkcolor=blue,pagebackref]{hyperref}
\usepackage[capitalise,nameinlink]{cleveref}
\usepackage{tikz-cd}

\newcommand{\C}{\mathbb C}
\newcommand{\Q}{\mathbb Q}
\newcommand{\R}{\mathbb R}
\newcommand{\Z}{\mathbb Z}
\renewcommand{\H}{\operatorname H}
\renewcommand{\P}{\mathbb P}
\newcommand{\cA}{\mathcal A}
\newcommand{\cO}{\mathcal O}
\newcommand{\cL}{\mathcal L}

\newcommand{\vol}{\operatorname{vol}}
\newcommand{\codim}{\operatorname{codim}}
\newcommand{\BG}{\operatorname{BG}}
\newcommand{\Sym}{\operatorname{Sym}}
\newcommand{\IC}{\operatorname{IC}}

\newtheorem{Thm}{Theorem}[section]
\newtheorem{Prop}[Thm]{Proposition}
\newtheorem{Lem}[Thm]{Lemma}
\newtheorem{Cor}[Thm]{Corollary}
\newtheorem{Ques}{Question}

\theoremstyle{definition}
\newtheorem{Def}[Thm]{Definition}
\newtheorem{Rem}[Thm]{Remark}

\crefname{Thm}{Theorem}{Theorems}
\crefname{Prop}{Proposition}{Propositions}
\crefname{Lem}{Lemma}{Lemmas}
\crefname{Cor}{Corollary}{Corollaries}
\crefname{Def}{definition}{definitions}
\crefname{Rem}{Remark}{Remarks}
\crefname{Ques}{Question}{Questions}
\crefname{Prob}{Problem}{Problems}
\crefname{Ex}{Example}{Examples}

\author[\'Angel David R\'ios Ortiz]{\'Angel David R\'ios Ortiz}
\address{Instituto de Matem\'aticas,
Universidad Nacional Aut\'onoma de M\'exico,\newline
Unidad Cuernavaca,
Av. Universidad s/n,
Col. Lomas de Chamilpa,
62210 Cuernavaca,
Morelos,
M\'exico}
\email{david.rios@im.unam.mx}
\title[Bogomolov--Guan manifolds are not formal]
{Bogomolov--Guan manifolds are not formal}

\begin{document}
\begin{abstract}
Bogomolov--Guan manifolds form the only known series of compact simply connected non-K\"ahler holomorphic symplectic manifolds. We prove that all Bogomolov--Guan manifolds of complex dimension at least four are nonformal over $\Q$.

The proof constructs a uniform degree-eight $a$-Massey product in a finite-dimensional invariant de Rham model of the symmetric quotient of the Bogomolov--Guan precursor. Its indeterminacy vanishes, and its nontriviality is detected by an explicit top-degree pairing arising from an $S_{n+1}$-invariant tensor contraction. A nonzero-degree map then transfers nonformality to the smooth Bogomolov--Guan manifold.

We also compare the Bogomolov--Guan precursor with the nilmanifold appearing in Guan's original construction, showing that they are related by a natural $S_{n+1}$-equivariant finite cover of degree $(n+1)^2$ and have the same invariant de Rham model. Using Guan's resolution, which we prove to be projective and semismall, together with Saito's decomposition theorem, we further show that the third Betti number of every Bogomolov--Guan manifold of complex dimension at least four is equal to $3$.
\end{abstract}
\maketitle

\section{Introduction}\label{sec:introduction}

One of the fundamental results of Deligne--Griffiths--Morgan--Sullivan is that every compact K\"ahler manifold is formal \cite{DGMS1975}. In the language of Sullivan's rational homotopy theory \cite{Sullivan1977}, this means that its de Rham algebra is connected to its cohomology algebra, endowed with the zero differential, by a chain of quasi-isomorphisms. Thus a substantial part of the rational homotopy type of a compact K\"ahler manifold is already encoded in its cohomology ring. Formality is therefore a strong topological shadow of K\"ahler geometry: it is a property of the underlying smooth manifold and is independent of the particular complex structure used to exhibit it as K\"ahler.

The interaction between formality, simple connectivity, and symplectic geometry is especially striking in low dimensions. A theorem of Miller \cite{Miller1979} asserts more generally that a $(k-1)$-connected compact manifold of dimension at most $4k-2$ is formal. In particular, every simply connected compact manifold of real dimension at most six is formal. Nonformal simply connected manifolds therefore first appear in real dimension seven. In the symplectic category the dimension is necessarily even, so real dimension eight is the first dimension in which one can hope to find a simply connected compact symplectic manifold which is not formal. Fern\'andez and Mu\~noz achieved precisely this in \cite{FernandezMunoz2008}, producing a nonformal simply connected compact symplectic eight-manifold. Their construction was motivated in part by the topology of Guan's non-K\"ahler holomorphic symplectic examples.

Bogomolov--Guan manifolds place this problem in a substantially more rigid geometric setting. Guan's constructions \cite{GuanII1995,GuanIII1995}, together with Bogomolov's geometric reinterpretation using primary Kodaira surfaces \cite{Bogomolov1996}, produce simply connected compact non-K\"ahler holomorphic symplectic manifolds in every even complex dimension at least four. Despite their non-K\"ahler nature, several aspects of their geometry parallel the hyper-K\"ahler case: their deformation theory satisfies a local Torelli theorem, and their second cohomology carries a Beauville--Bogomolov--Fujiki-type quadratic form satisfying a Fujiki relation \cite{KurnosovVerbitsky2019}. Further geometric and topological properties have been developed in \cite{BKKY2022,AKV2026}. It is therefore natural to ask whether formality, which is automatic in the compact K\"ahler case, survives in this non-K\"ahler holomorphic symplectic setting.

A recent paper of Guan \cite{Guan2025} proposes a relative Massey-product argument in complex dimension four and a higher-dimensional extension. The proof given here is independent of those calculations. In fact, the higher-dimensional defining system written there vanishes for dimensional reasons once $n>2$; see \cref{rem:central-input}. Our main result establishes nonformality uniformly in every dimension and for every deformation class arising from the Bogomolov--Guan construction.

We now fix the geometric notation. Following Bogomolov \cite{Bogomolov1996}, let $S=\cL^{\times}/\langle\lambda\rangle$ be a primary Kodaira surface, where $\cL$ is a line bundle of nonzero degree $d$ on an elliptic curve $E$ and $|\lambda|>1$. Its elliptic fiber is $F=\C^{\times}/\langle\lambda\rangle$. Put $m=n+1$ and assume that $m\mid d$; write $d=mt$ with $t\in\Z\setminus\{0\}$. Consider
\[
W=\left(S^{[m]}\longrightarrow \Sym^mE\xrightarrow{\ \Sigma\ }E\right)^{-1}(0).
\]
Bogomolov's construction associates to these data a smooth simply connected compact holomorphic symplectic manifold $Q$ of complex dimension $2n$, together with a finite morphism $Q\to W/F$ of degree $m^2$; see \cite{Bogomolov1996} and \cite[Theorem~3.3]{BKKY2022}.

Alongside the Hilbert-scheme construction we use the product-level space
\[
R=\left\{(s_1,\ldots,s_m)\in S^m:\sum_{i=1}^m\pi(s_i)=0\right\},
\qquad
N_n:=R/F,
\]
where $F$ acts diagonally. In the terminology of Abubakarova--Kuznetsova--Verbitsky \cite[\S1.2]{AKV2026}, $N_n$ is the \emph{restricted Bogomolov--Guan precursor}, or \emph{BG-precursor}. Since no unrestricted precursor occurs below, we call $N_n$ simply the \emph{Bogomolov--Guan precursor}. It is a compact complex nilmanifold \cite[Claim~4.6]{AKV2026}, naturally equipped with an $S_m$-action, and we put
\[
X_n:=N_n/S_m.
\]

There is a second, closely related nilmanifold which must be kept distinct from the precursor. Guan's original construction \cite{GuanII1995,GuanIII1995} uses a nilmanifold $M_{n,t}$ with an $S_m$-action and the quotient
\[
Y_{n,t}:=M_{n,t}/S_m.
\]
The two nilmanifolds are commensurable but their natural integral lattices are different. More precisely, the precursor uses the quotient, or weight, lattice $P_n=\Z^m/\Z(1,\ldots,1)$ in the elliptic-fiber directions, whereas Guan's nilmanifold uses the root lattice $A_n=\{(a_i)\in\Z^m:\sum a_i=0\}$. Since $[P_n:A_n]=m$, the natural $S_m$-equivariant map $M_{n,t}\to N_n$ has degree $m^2$. We prove this comparison carefully in \cref{prop:guan-precursor-cover}. It induces the diagram
\[
\begin{tikzcd}
M_{n,t} \arrow[r] \arrow[d] & N_n \arrow[d] \\
Y_{n,t}=M_{n,t}/S_m \arrow[r] & X_n=N_n/S_m
\end{tikzcd}
\]
with horizontal maps finite of degree $m^2$. The distinction is important: $X_n$ is the quotient on which we construct the nonformality obstruction, while $Y_{n,t}$ is the quotient admitting Guan's degree-one Hilbert-type resolution by the smooth Bogomolov--Guan manifold.

\begin{Def}
A compact complex manifold $X$ of complex dimension $2n$ is called a \emph{Bogomolov--Guan manifold of type $\BG_n$}, or simply a \emph{$\BG_n$-manifold}, if it is deformation equivalent to one of the distinguished manifolds $Q_{n,t}$ arising from the Guan--Bogomolov construction, for some $t\neq0$. We use $\BG_n$ for the collection of deformation classes obtained in this way.
\end{Def}
\begin{Rem}
For $n=1$ the resulting surface is a K3 surface; see \cite[Example~3.5]{BKKY2022}. Thus the first genuinely non-K\"ahler Bogomolov--Guan manifolds occur in complex dimension four.
\end{Rem}

The first and main theorem of this paper is the following.

\begin{Thm}\label{thm:main}
Every $\BG_n$-manifold, for $n\geq2$, is nonformal over $\Q$, $\R$, and $\C$.
\end{Thm}

We also compute the first nontrivial odd Betti number uniformly, yielding our second result.

\begin{Thm}\label{thm:intro-b3}
Let $X$ be a $\BG_n$-manifold with $n\geq2$. Then
\[
b_3(X)=3.
\]
In particular, the underlying smooth manifold of $X$ admits no K\"ahler complex structure and, more generally, no complex structure satisfying the $\partial\bar\partial$-lemma.
\end{Thm}

\subsection{Strategy of proof}

The proof separates the homotopy-theoretic obstruction from the semismall calculation of $b_3$. The Bogomolov--Guan precursor $N_n$ is a compact complex nilmanifold with an $S_m$-action. By Nomizu's theorem and exactness of finite-group invariants, the symmetric quotient $X_n=N_n/S_m$ is modeled over $\C$ by the finite-dimensional cdga $\cA_n=(E_n,d)^{S_m}$, where $E_n$ is generated in degree one by four copies of the standard representation of $S_m$.

Inside $\cA_n$ we construct the degree-eight $a$-Massey product $\langle b;a,b,c\rangle$. Its indeterminacy vanishes because $\H^3(\cA_n)=0$, and it is represented by $M_n=a\beta\gamma+b\gamma\alpha+c\alpha\beta$. The main algebraic calculation is the uniform identity
\[
M_n a^{n-2}q^{n-2}=(-1)^n(n+1)^2n!(n-2)!\,\vol_n.
\]
Since the top form on the right is not exact, $[M_n]\neq0$ for every $n\geq2$, and hence $X_n$ is nonformal. Bogomolov's Hilbert-scheme construction gives a map of nonzero degree $Q\to X_n$. Since $X_n$ is a rational Poincar\'e duality space, the domination theorem of Milivojevi\'c--Stelzig--Zoller \cite[Theorem~A]{MSZ2023} transfers nonformality to $Q$, and therefore to every manifold deformation equivalent to $Q$.

The computation of $b_3$ uses the other quotient. For $d=mt$, Guan's nilmanifold $M_{n,t}$ is an $m^2$-fold cover of the precursor but has the same $S_m$-equivariant real Lie algebra, hence the same invariant de Rham model after extension of scalars. Consequently
\[
b_2(Y_{n,t})=5,\qquad b_3(Y_{n,t})=0.
\]
Guan's original construction gives a degree-one Hilbert-type resolution $f:Q_{n,t}\to Y_{n,t}$. We prove that it is projective and semismall and then apply Saito's decomposition theorem. In degrees two and three the only additional contribution comes from the codimension-two transposition sector. If $Z_{n,t}$ denotes the normalization of its closure, then
\[
\H^3(Q_{n,t},\C)\cong \H^1(Z_{n,t},\C),
\]
and an explicit fixed-locus calculation gives $b_1(Z_{n,t})=3$.

\subsection{Relation with previous work}

The obstruction used here is the $a$-Massey product introduced by Cavalcanti--Fern\'andez--Mu\~noz \cite{CFM2008}, building on the nonformal symplectic construction of Fern\'andez--Mu\~noz \cite{FernandezMunoz2008}. The formulas proposed in \cite{Guan2025} motivate a related dimension-dependent defining system; as explained in \cref{rem:central-input}, its stated higher-dimensional representative vanishes identically for $n>2$. Our construction instead keeps the central class equal to $b$ in every dimension and detects the resulting degree-eight class by a complementary top-degree pairing.

A second point of the present paper is the separation between the modern Bogomolov--Guan precursor and Guan's original nilmanifold. They have the same $S_m$-equivariant Lie algebra over $\R$ but different natural integral lattices, related by an index-$m$ root-to-weight inclusion in each real elliptic-fiber lattice. The resulting degree-$m^2$ cover explains both why their invariant structure equations agree and why they should not be identified as the same compact nilmanifold. The nonformality argument is naturally carried out on the precursor quotient $X_n$, whereas the semismall resolution used for $b_3$ is naturally a resolution of $Y_{n,t}$.

 \subsection*{Acknowledgements}
Initial attempts to prove some of the results of this paper were carried out in collaboration with Tommaso Sferruzza. I am grateful to him for the discussions and ideas that helped shape the early stages of this project. Over the years, I have also benefited from many conversations on Bogomolov--Guan manifolds and related questions. I would particularly like to thank Giuseppe Barbaro, Filippo Fagioli, Nikon Kurnosov and Vasily Rogov for several useful and stimulating discussions. I also thank Misha Verbitsky, whose comment on looking at the Bogomolov--Guan precursor indirectly helped me find a mistake in the initial computations of this paper, and Nikon, who confirmed that the computation was indeed wrong. Finally, I am especially grateful to Ángel Cano, Pepe Seade and Jawad Snoussi for their support and encouragement, and for making me feel so welcome in Cuernavaca.

\subsection*{AI disclosure}
OpenAI's GPT-5.6 Sol was used as an auxiliary tool in the preparation of this manuscript, primarily for language editing, improvements to the exposition, and assistance in checking and refining some arguments. All mathematical statements, proofs, and conclusions were independently verified by the author, who takes full responsibility for the content of the paper.

\section{Formality, Massey products, and finite quotients}
\label{sec:preliminaries}

Recall that a connected cdga $(A,d)$ over a field of characteristic zero is \emph{formal} if it is connected to $(\H^*(A),0)$ by a zigzag of quasi-isomorphisms. A connected topological space $X$ is called \emph{rationally formal} if its Sullivan algebra $A_{\mathrm{PL}}(X;\Q)$ is formal \cite[\S12]{Sullivan1977}.

For a smooth manifold $M$, the de Rham algebra is a real model of $A_{\mathrm{PL}}(M)$, so rational formality is equivalent to formality of its de Rham algebra after extension of scalars, cf. \cite[Remark~2.86 and Proposition~2.101]{FelixOpreaTanre2008}. Moreover, formality descends along field extensions of characteristic zero \cite[Theorem~12.1]{Sullivan1977}, see also \cite[Remark~3.7]{MSZ2023}. We shall therefore perform all explicit
calculations over $\C$.

\begin{Def}\label{def:a-massey}
Let $(A,d)$ be a cdga and let $b,a_1,a_2,a_3\in A^2$ be closed elements such that each product $ba_i$ is exact. Choose $\xi_i\in A^3$ with $d\xi_i=ba_i$. The \emph{third-order
$a$-Massey product} with central element $b$ is
\begin{equation}\label{eq:a-massey-product}
\langle b;a_1,a_2,a_3\rangle
=
\left\{
\left[
a_1\xi_2\xi_3+
a_2\xi_3\xi_1+
a_3\xi_1\xi_2
\right]:
d\xi_i=ba_i
\right\}
\subset \H^8(A).
\end{equation}
The choice $(\xi_1,\xi_2,\xi_3)$ is called a \emph{defining system},
and the product is \emph{nontrivial} if
$0\notin\langle b;a_1,a_2,a_3\rangle$.
\end{Def}

This is the specialization to degree-two classes of the $a$-Massey product introduced in \cite[Definition~2.3]{CFM2008}. The representative in \eqref{eq:a-massey-product} is closed \cite[Proposition~2.2]{CFM2008}, and the product depends only on the input cohomology classes and is invariant under quasi-isomorphisms \cite[Lemma~2.5, Remark~2.8 and Lemma~2.9]{CFM2008}. Most importantly,
\[
0\notin\langle b;a_1,a_2,a_3\rangle
\quad\Longrightarrow\quad
A\ \text{is nonformal}
\]
by \cite[Theorem~2.10]{CFM2008}.

Thus exhibiting one nonzero representative is not enough: one must also control the indeterminacy and rule out a defining system producing the zero class.

We shall also use the following standard fact about finite quotients. We include the multiplicative statement because the action considered
later has nonisolated fixed loci.

\begin{Lem}\label{lem:finite-quotient}
Let a finite group $G$ act smoothly on a compact connected manifold $N$. Then
\[
\Omega^*(N;\C)^G
\]
is a cdga model of the underlying topological space $N/G$. If $N$ is closed and oriented and the action preserves orientation, then $N/G$ is a rational Poincar\'e duality space of dimension
$\dim_{\R}N$.
\end{Lem}
\begin{proof}
By Illman's equivariant triangulation theorem \cite{Illman1978}, $N$ admits a $G$-equivariant triangulation $K$. After barycentric subdivision we may assume that the action is without inversions, so $K/G$ is a simplicial complex with realization $N/G$.

Let $A_{\mathrm{ps}}(K;\C)$ denote the cdga of compatible piecewise smooth forms. The de Rham and polynomial de Rham comparison maps give a $G$-equivariant zigzag of quasi-isomorphisms
\[
\Omega^*(N;\C)
\longrightarrow
A_{\mathrm{ps}}(K;\C)
\longleftarrow
A_{\mathrm{PL}}(K;\C);
\]
see \cite[Theorem~7.1]{Sullivan1977} and \cite[Theorem~2.2]{BousfieldGugenheim1976}. Since $G$ is finite, taking invariants is exact. Moreover, the absence of inversions gives $A_{\mathrm{PL}}(K;\C)^G\cong A_{\mathrm{PL}}(K/G;\C)$. Hence $\Omega^*(N;\C)^G$ is connected to
$A_{\mathrm{PL}}(N/G;\C)$ by a zigzag of quasi-isomorphisms and is
therefore a cdga model of $N/G$. The same argument over $\Q$ gives $\H^*(N/G;\Q)\cong \H^*(N;\Q)^G$. Suppose now that $N$ is closed and oriented of real dimension $d$ and that $G$ preserves its orientation. For $0\neq z\in \H^k(N;\Q)^G$, choose $w\in \H^{d-k}(N;\Q)$ with $\langle z\smile w,[N]\rangle\neq0$. Averaging,
\[
\overline w=\frac1{|G|}\sum_{g\in G}g^*w
\]
is $G$-invariant and still satisfies $\langle z\smile\overline w,[N]\rangle\neq0$, because both $z$ and $[N]$ are $G$-invariant. Thus the cup-product pairing on $\H^*(N;\Q)^G$ is nondegenerate. Since $\H^d(N;\Q)^G=\H^d(N;\Q)\cong\Q$, the quotient $N/G$ therefore satisfies rational Poincar\'e duality.
\end{proof}

\begin{Rem}
The orbifold discussion in \cite[\S2.3]{CFM2008} assumes isolated singularities, whereas \cref{lem:finite-quotient} applies to the nonisolated fixed loci occurring here. Throughout the paper, cohomology of a finite quotient means ordinary cohomology of its underlying topological space, not Chen--Ruan cohomology.
\end{Rem}

\section{The invariant algebra and its defining system}
\label{sec:algebra}

We now introduce the finite-dimensional differential graded algebra in which the nonformality obstruction will be constructed. Put $m=n+1$, and let
\[
V=\{(z_1,\ldots,z_m)\in\C^m:\sum_{i=1}^m z_i=0\}
\]
be the standard representation of $S_m$. Thus $\dim V=n$. The standard $S_m$-invariant bilinear form identifies $V$ with $V^\vee$, and we use this identification throughout. Let
\[
E_n=\Lambda\bigl(V_x^\vee\oplus V_w^\vee\oplus V_u^\vee\oplus V_v^\vee\bigr),
\]
with all generators placed in degree one. We write the four families as $x=(x_i)$, $w=(w_i)$, $u=(u_i)$ and $v=(v_i)$. Each satisfies the relation $\sum_i x_i=\sum_i w_i=\sum_i u_i=\sum_i v_i=0$.

We shall repeatedly use the $S_m$-invariant contractions
\begin{equation}\label{eq:uniform-tensors}
B(r,s)=\sum_{i=1}^m r_is_i,
\qquad
T_n(r,s,t)=\sum_{i=1}^m r_is_it_i,
\end{equation}
where all products are exterior products taken in the displayed order. Define
\begin{equation}\label{eq:uniform-quadrics}
a=B(x,u),\qquad
b=B(x,w),\qquad
c=B(w,u),\qquad
q=B(w,v).
\end{equation}

We equip $E_n$ with the $S_m$-equivariant differential
\begin{equation}\label{eq:uniform-d}
dx_i=dw_i=du_i=0,
\qquad
dv_i=b-mx_iw_i.
\end{equation}
This is compatible with the relations among the generators, since $\sum_i dv_i=mb-m\sum_i x_iw_i=0$, and $d^2=0$ because $db=0$. We shall work with the invariant cdga $\cA_n=(E_n,d)^{S_m}$. The elements $a,b,c$ are closed, and so is $q$. Indeed, $dq=-\sum_iw_i(b-mx_iw_i)=0$, because $\sum_iw_i=0$ and every term $w_ix_iw_i$ vanishes.

We now construct a defining system for the $a$-Massey product of \cref{def:a-massey}. Set
\begin{equation}\label{eq:uniform-primitives}
\alpha=T_n(x,u,v),\qquad
\beta=T_n(x,w,v),\qquad
\gamma=T_n(w,u,v).
\end{equation}
Using \eqref{eq:uniform-d}, the terms involving $mx_iw_i$ vanish because they introduce a repeated $x_i$ or $w_i$. Hence
\begin{equation}\label{eq:uniform-primitive-identities}
d\alpha=ba,\qquad
d\beta=b^2,\qquad
d\gamma=bc.
\end{equation}
Thus $\alpha,\beta,\gamma$ form a defining system for $\langle b;a,b,c\rangle$. The corresponding degree-eight representative is
\begin{equation}\label{eq:uniform-M}
M_n=a\beta\gamma+b\gamma\alpha+c\alpha\beta.
\end{equation}
In particular, $M_n$ is closed and $S_m$-invariant.

The crucial feature is the uniformity of the construction: the central class is always $b$, and the obstruction $M_n$ always lies in degree eight. The dependence on $n$ will enter only through the complementary closed class used to detect the nonvanishing of $[M_n]$.

\subsection{Low-degree cohomology and indeterminacy}

We next compute the low-degree cohomology needed to control the indeterminacy of the $a$-Massey product. Since $V^{S_m}=0$, there are no invariant degree-one elements, so $\cA_n^1=0$.

The standard representation is irreducible and self-dual, and its unique invariant bilinear form is symmetric. Hence $(\Lambda^2V)^{S_m}=0$, while $(V\otimes V)^{S_m}=\C\cdot B$. It follows that the invariant quadratic elements are precisely the six pairings between distinct copies of $V$. Besides $a,b,c,q$, set $p=B(x,v)$ and $r=B(u,v)$. From \eqref{eq:uniform-d} one gets
\begin{equation}\label{eq:degree-two-differentials}
da=db=dc=dp=dq=0, \qquad dr=m\,T_n(x,w,u).
\end{equation}
Since $T_n(x,w,u)\neq0$ for $n\geq2$, the class $r$ is not closed. As there are no degree-two boundaries, we obtain
\begin{equation}\label{eq:low-cohomology}
\H^1(\cA_n)=0, \qquad \H^2(\cA_n)
= \operatorname{span}\{[a],[b],[c],[p],[q]\} \cong\C^5.
\end{equation}

The vanishing needed to eliminate the indeterminacy is the following.

\begin{Lem}\label{lem:uniform-H3}
For every $n\geq2$, one has $\H^3(\cA_n)=0$.
\end{Lem}
\begin{proof}
Let $P=\C^m=\C\oplus V$ be the permutation representation. For $m\geq3$, the diagonal action of $S_m$ on ordered triples of indices has five orbits: one with all three indices equal, three with exactly two equal, and one with all three distinct. Therefore $\dim(P^{\otimes3})^{S_m}=5$.

Expanding $(\C\oplus V)^{\otimes3}$ and using $V^{S_m}=0$ together with $\dim(V^{\otimes2})^{S_m}=1$, we obtain $\dim(V^{\otimes3})^{S_m}=1$. This invariant line is generated by the symmetric tensor $T_n$. It is nonzero, for instance because its diagonal cubic evaluates nontrivially at $(1,1,-2,0,\ldots,0)\in V$.

We now decompose $\Lambda^3(V_x\oplus V_w\oplus V_u\oplus V_v)$ according to multidegree in the four copies. An invariant of type $(3,0,0,0)$ would require a totally alternating invariant tensor in $V^{\otimes3}$, while one of type $(2,1,0,0)$ would require antisymmetrization in two tensor factors. Both vanish because the unique invariant tensor in $V^{\otimes3}$ is symmetric. Thus invariant cubic forms occur only when the three factors come from three distinct copies.

There are four such choices. Hence $\cA_n^3$ has basis
\[
t:=T_n(x,w,u),\qquad \alpha,\qquad \beta,\qquad \gamma.
\]
By \eqref{eq:uniform-primitive-identities}, the differentials of $\alpha,\beta,\gamma$ are $ab$, $b^2$, and $bc$, respectively. These have distinct multidegrees $(2,1,1,0)$, $(2,2,0,0)$, and $(1,2,1,0)$, and each is nonzero for $n\geq2$. They are therefore linearly independent. Hence
\[
\ker\bigl(d:\cA_n^3\to\cA_n^4\bigr)=\C\cdot t.
\]
On the other hand, \eqref{eq:degree-two-differentials} gives $dr=mt$. Thus $t$ is exact, and therefore $\H^3(\cA_n)=0$.
\end{proof}

The preceding vanishing removes the entire indeterminacy of the $a$-Massey product.

\begin{Cor}\label{cor:zero-indeterminacy}
For every $n\geq2$,
\begin{equation}\label{eq:singleton-massey}
\langle b;a,b,c\rangle=\{[M_n]\}.
\end{equation}
\end{Cor}

\begin{proof}
Any two primitives of $ba$, $b^2$, or $bc$ differ by a closed degree-three element, which is exact by \cref{lem:uniform-H3}. For example, replacing $\alpha$ by $\alpha+d\eta$, with $|\eta|=2$, changes the representative by
\[
M_n'-M_n=d\eta\,(c\beta-b\gamma).
\]
The factor $c\beta-b\gamma$ is closed, since $d(c\beta-b\gamma)=cb^2-b^2c=0$. Hence $M_n'-M_n=d\bigl(\eta(c\beta-b\gamma)\bigr)$. The same argument applies to changes of $\beta$ and $\gamma$, so every defining system determines the same cohomology class. Independence of the chosen representatives of the input classes follows from \cite[Lemma~2.5]{CFM2008}.
\end{proof}

\section{A nonzero pairing in every dimension}
\label{sec:pairing}

We now prove a uniform top-degree identity detecting the Massey class
$[M_n]$. The computation has three ingredients: an orthonormalization
of the invariant quadratic form, two elementary contractions of the
invariant cubic tensor, and the resulting top-degree calculation.

\begin{Lem}\label{lem:uniform-pairing}
Using $x_1,\ldots,x_n$, and similarly $w_i,u_i,v_i$, as independent
coordinates on the four copies of $V$, set
\[
\vol_n=
x_1\cdots x_n\,
w_1\cdots w_n\,
u_1\cdots u_n\,
v_1\cdots v_n.
\]
Then, for every $n\geq2$,
\begin{equation}\label{eq:uniform-pairing}
M_n\,a^{n-2}q^{n-2}
=
(-1)^n(n+1)^2n!(n-2)!\,\vol_n.
\end{equation}
In particular, $M_n\,a^{n-2}q^{n-2}\neq0$.
\end{Lem}

\subsection*{Step 1: orthonormal coordinates}

\begin{Lem}\label{lem:orthogonal-coordinates}
There are degree-one generators $X_i,W_i,U_i,V_i$, for
$1\leq i\leq n$, such that
\begin{equation}\label{eq:orthogonal-quadrics}
a=\sum_iX_iU_i,\qquad
b=\sum_iX_iW_i,\qquad
c=\sum_iW_iU_i,\qquad
q=\sum_iW_iV_i.
\end{equation}
Moreover,
\begin{equation}\label{eq:volume-change}
X_1\cdots X_n\,
W_1\cdots W_n\,
U_1\cdots U_n\,
V_1\cdots V_n
=
(n+1)^2\vol_n.
\end{equation}
\end{Lem}

\begin{proof}
Since the coordinates of $V$ satisfy $\sum_{i=1}^m r_i=0$, with
$m=n+1$, we may take $r_1,\ldots,r_n$ as independent and write
$r_m=-\sum_{i=1}^n r_i$. In these coordinates the invariant bilinear
form $B$ has matrix
\[
G=I_n+\mathbf 1\mathbf 1^{\,t},
\]
where $\mathbf 1=(1,\ldots,1)^t$. Hence $G$ has eigenvalue $1$ on
$\mathbf 1^\perp$ and eigenvalue $n+1$ on $\C\mathbf 1$. In
particular, $\det G=n+1$.

Since $G$ is positive definite, applying the change of coordinates
$G^{1/2}$ simultaneously to the four copies $x,w,u,v$ transforms $B$
into the standard pairing. This gives \eqref{eq:orthogonal-quadrics}.
Moreover, $\det(G^{1/2})=\sqrt{n+1}$, so the four top exterior
products together acquire the factor $(n+1)^2$, proving
\eqref{eq:volume-change}.
\end{proof}

\subsection*{Step 2: the invariant cubic tensor}

Write the original coordinates in the orthonormal basis as
$x_i=\sum_a e_{ia}X_a$, and use the same coefficients for the other
three copies. Thus $w_i=\sum_a e_{ia}W_a$,
$u_i=\sum_a e_{ia}U_a$, and $v_i=\sum_a e_{ia}V_a$.

The vectors $e_i=(e_{i1},\ldots,e_{in})$ satisfy
\begin{equation}\label{eq:e-relations}
\sum_{i=1}^m e_{ia}e_{ib}=\delta_{ab},
\qquad
\sum_{i=1}^m e_{ia}=0,
\qquad
\langle e_i,e_j\rangle
=
\delta_{ij}-\frac1{n+1}.
\end{equation}
Define
\begin{equation}\label{eq:cubic-coefficients}
C_{abc}=\sum_{i=1}^m e_{ia}e_{ib}e_{ic}.
\end{equation}
Then $T_n(R,S,T)=\sum_{a,b,c}C_{abc}R_aS_bT_c$ in orthonormal
coordinates. In particular, this expression gives $\alpha,\beta$ and
$\gamma$ by taking $(R,S,T)=(X,U,V),(X,W,V)$ and $(W,U,V)$,
respectively.

\begin{Lem}\label{lem:cubic-contractions}
The tensor $C_{abc}$ is symmetric and satisfies
\begin{equation}\label{eq:cubic-contractions}
\sum_aC_{aac}=0
\quad\text{for every }c,
\qquad
\sum_{a,b,c}C_{abc}^2
=
\frac{n(n-1)}{n+1}.
\end{equation}
\end{Lem}

\begin{proof}
Symmetry is immediate from \eqref{eq:cubic-coefficients}. Using
\eqref{eq:e-relations},
\[
\sum_aC_{aac}
=
\sum_i\langle e_i,e_i\rangle e_{ic}
=
\frac{n}{n+1}\sum_i e_{ic}
=
0.
\]
For the norm, we obtain
\[
\sum_{a,b,c}C_{abc}^2
=
\sum_{i,j=1}^m\langle e_i,e_j\rangle^3.
\]
There are $m$ diagonal terms, each equal to $(n/m)^3$, and
$m(m-1)$ off-diagonal terms, each equal to $(-1/m)^3$. Since
$m=n+1$, their sum is $n(n-1)/(n+1)$.
\end{proof}

\subsection*{Step 3: the top-degree contraction}

Set
\begin{equation}\label{eq:orthogonal-volume}
\Omega=
(-1)^n
X_1\cdots X_n\,
W_1\cdots W_n\,
U_1\cdots U_n\,
V_1\cdots V_n
\end{equation}
and write
\[
\kappa_n:=\sum_{a,b,c}C_{abc}^2
=
\frac{n(n-1)}{n+1}.
\]

\begin{Lem}\label{lem:three-top-contractions}
In the coordinates of \cref{lem:orthogonal-coordinates},
\begin{align}
a\beta\gamma\,a^{n-2}q^{n-2}
&=(n-1)!(n-2)!\,\kappa_n\,\Omega,
\label{eq:first-top-contraction}\\
b\gamma\alpha\,a^{n-2}q^{n-2}
&=((n-2)!)^2\,\kappa_n\,\Omega,
\label{eq:second-top-contraction}\\
c\alpha\beta\,a^{n-2}q^{n-2}
&=((n-2)!)^2\,\kappa_n\,\Omega.
\label{eq:third-top-contraction}
\end{align}
\end{Lem}

\begin{proof}
Recall that
$a=\sum_iX_iU_i$ and $q=\sum_iW_iV_i$. After expanding their
powers, a contribution to the top degree can occur only when every
$X$-index is paired with the corresponding $U$-index and every
$W$-index with the corresponding $V$-index. With the sign convention
incorporated in the definition of $\Omega$, one free $X$--$U$ pair
produces a Kronecker symbol $\delta_{ij}$, while two such pairs produce
the antisymmetrized contraction
$\delta_{ij}\delta_{lr}-\delta_{ir}\delta_{lj}$; the same rule applies
to the $W$--$V$ pairs. The powers of $a$ and $q$ account for the
remaining factorials.

We treat the three summands separately. Using
$\beta=\sum C_{ijk}X_iW_jV_k$ and
$\gamma=\sum C_{lrs}W_lU_rV_s$, the tensor contraction associated
with $a\beta\gamma$ is
\[
-\sum_{i,j,k,l,r,s}
C_{ijk}C_{lrs}\,
\delta_{ir}
\bigl(
\delta_{jk}\delta_{ls}
-
\delta_{js}\delta_{lk}
\bigr).
\]
The first term contains traces of $C$ and therefore vanishes by
\cref{lem:cubic-contractions}. The second gives
\[
\sum_{i,j,k}C_{ijk}C_{kij}
=
\sum_{i,j,k}C_{ijk}^2
=
\kappa_n,
\]
using the symmetry of $C$. Since the explicit factor $a$ combines
with $a^{n-2}$ to give $a^{n-1}$, the remaining quadratic factors
contribute $(n-1)!(n-2)!$. This proves
\eqref{eq:first-top-contraction}.

For the second summand, write
$b=\sum_pX_pW_p$. Since
$\gamma=\sum C_{ijk}W_iU_jV_k$ and
$\alpha=\sum C_{lrs}X_lU_rV_s$, the two pairs of free indices in the
$X$--$U$ and $W$--$V$ sectors give
\[
-\sum_{p,i,j,k,l,r,s}
C_{ijk}C_{lrs}
\bigl(
\delta_{pj}\delta_{lr}
-
\delta_{pr}\delta_{lj}
\bigr)
\bigl(
\delta_{pk}\delta_{is}
-
\delta_{ps}\delta_{ik}
\bigr).
\]
Expanding the two determinants gives four terms. Three contain a trace
of $C$:
\[
\sum C_{ipp}C_{lli},
\qquad
\sum C_{ipi}C_{llp},
\qquad
\sum C_{iji}C_{jpp},
\]
and hence vanish. The remaining term is
\[
\sum_{i,j,p}C_{ijp}C_{jpi}
=
\sum_{i,j,p}C_{ijp}^2
=
\kappa_n.
\]
There remain $n-2$ copies of both $a$ and $q$, so their expansion
contributes $((n-2)!)^2$. This proves
\eqref{eq:second-top-contraction}.

Finally, write $c=\sum_pW_pU_p$. Using
$\alpha=\sum C_{ijk}X_iU_jV_k$ and
$\beta=\sum C_{lrs}X_lW_rV_s$, the corresponding contraction is
\[
-\sum_{p,i,j,k,l,r,s}
C_{ijk}C_{lrs}
\bigl(
\delta_{ip}\delta_{lj}
-
\delta_{ij}\delta_{lp}
\bigr)
\bigl(
\delta_{pk}\delta_{rs}
-
\delta_{ps}\delta_{rk}
\bigr).
\]
Again three of the four terms contain a trace of $C$ and vanish. The
only surviving contraction is
\[
\sum_{p,j,k}C_{pjk}C_{jkp}
=
\sum_{p,j,k}C_{pjk}^2
=
\kappa_n.
\]
As in the preceding case, the remaining powers of $a$ and $q$
contribute $((n-2)!)^2$. This proves
\eqref{eq:third-top-contraction}.
\end{proof}

\begin{proof}[Proof of \cref{lem:uniform-pairing}]
By \eqref{eq:uniform-M} and
\cref{lem:three-top-contractions},
\begin{align*}
M_n\,a^{n-2}q^{n-2}
&=
\left((n-1)!(n-2)!+2((n-2)!)^2\right)\kappa_n\,\Omega\\
&=
n!(n-2)!\,\Omega.
\end{align*}
Using \eqref{eq:orthogonal-volume} and
\eqref{eq:volume-change}, we have
$\Omega=(-1)^n(n+1)^2\vol_n$. Therefore
\[
M_n\,a^{n-2}q^{n-2}
=
(-1)^n(n+1)^2n!(n-2)!\,\vol_n,
\]
as claimed.
\end{proof}

We can now conclude the algebraic part of the argument.

\begin{Thm}\label{thm:uniform-algebra}
For every $n\geq2$, the cdga $\cA_n$ is nonformal. More precisely,
\[
\langle b;a,b,c\rangle=\{[M_n]\},
\qquad
[M_n]\neq0.
\]
\end{Thm}

\begin{proof}
The complementary factor $a^{n-2}q^{n-2}$ is closed. We claim that the nonzero top form in \eqref{eq:uniform-pairing} is not exact even in $E_n$.

Choose independent bases in the four copies of $V$. A monomial of degree $4n-1$ omits exactly one basis generator. When its differential is taken, the only possible contribution comes from differentiating a $v$-generator. This removes one $v$ and inserts a wedge of one $x$- and one $w$-generator. Since the original monomial omitted only one generator, at least one of these two inserted factors is already present, and the resulting wedge product vanishes. Hence $dE_n^{4n-1}=0$.

It follows from \eqref{eq:uniform-pairing} that $M_n$ cannot be exact: otherwise its product with the closed class $a^{n-2}q^{n-2}$ would be exact, contradicting the nonzero top-degree identity. By \cref{cor:zero-indeterminacy}, the $a$-Massey product is therefore the nontrivial singleton $\{[M_n]\}$. The obstruction theorem \cite[Theorem~2.10]{CFM2008} now gives the nonformality of $\cA_n$.

We have proved nonformality after extension of scalars to $\C$. If the rational Sullivan model of $X_n$ were formal over $\Q$, then its scalar extension to $\C$ would be formal as well. Hence $X_n$ is nonformal over $\Q$; the same argument gives nonformality over $\R$ and $\C$.
\end{proof}

\begin{Rem}\label{rem:central-input}
For $n=1$ the differential vanishes, so $\cA_1$ is formal. There is also a dimensional obstruction to the higher-dimensional extension considered in \cite{Guan2025}. In our notation, the proposed defining system is $b^{n-2}\alpha,b^{n-2}\beta,b^{n-2}\gamma$ for the product $\langle b^{n-1};a,b,c\rangle$, and its associated representative is $b^{2n-4}M_n$.

For $n>2$ this representative vanishes identically. Indeed, every term of $M_n$ has multidegree $(2,2,2,2)$ in the four copies $(x,w,u,v)$, while $b^{2n-4}$ contributes $2n-4$ additional factors from each of the $x$- and $w$-copies. Thus every term of $b^{2n-4}M_n$ contains $2n-2>n$ factors from each of these two $n$-dimensional copies, and hence
\[
b^{2n-4}M_n=0.
\]
Consequently this dimension-dependent defining system contains the zero class and does not yield a nonformality obstruction for $n>2$. In particular, the corresponding higher-dimensional nonvanishing calculation in \cite{Guan2025} cannot hold with the stated defining system.

Our construction avoids this dimensional obstruction by keeping the central element equal to $b$ in every dimension. The class $M_n$ therefore always has degree eight, while the dependence on $n$ is absorbed by the complementary closed factor $a^{n-2}q^{n-2}$.
\end{Rem}

\section{Passage to Bogomolov--Guan manifolds}
\label{sec:geometry}

We now relate the invariant algebra studied in the preceding sections to the smooth Bogomolov--Guan manifolds. Two finite symmetric quotients occur naturally. The quotient of the Bogomolov--Guan precursor is the space on which the nonformality obstruction is constructed. Guan's original nilmanifold has a different integral lattice and gives the quotient resolved by the smooth manifold. Keeping these two spaces separate is essential.

\subsection{The Bogomolov--Guan precursor and its symmetric quotient}

Write the primary Kodaira surface as an elliptic fibration $\pi:S\to E$ with elliptic fiber $F$, put $m=n+1$, and write $d=\deg\cL=mt$. Recall that
\[
R=
\left\{
(s_1,\ldots,s_m)\in S^m:
\sum_{i=1}^m\pi(s_i)=0
\right\},
\qquad
N_n=R/F,
\qquad
X_n=N_n/S_m,
\]
where $F$ acts diagonally and $S_m$ permutes the factors. Thus $N_n$ is the restricted Bogomolov--Guan precursor of \cite[\S1.2 and Definition~4.1]{AKV2026}. As above, we call it simply the Bogomolov--Guan precursor. The diagonal $F$-action is free. The projection to the base makes $R$ an $F^m$-bundle over $\ker(E^m\to E)\cong E^{m-1}$, and hence $N_n$ is an $F^{m-1}$-bundle over $E^{m-1}$. In particular, $N_n$ is compact and connected of complex dimension $2n$.

The fact that the precursor is a complex nilmanifold is observed in \cite[Claim~4.6]{AKV2026}. We record an $S_m$-equivariant presentation adapted to the de Rham calculation. A primary Kodaira surface admits a presentation $S=\Gamma_S\backslash G_S$, where $G_S\cong H_3(\R)\times\R$ is simply connected and two-step nilpotent and $\Gamma_S\subset G_S$ is a cocompact lattice \cite[\S2]{Hasegawa2005}. The elliptic fibration is induced by a surjective homomorphism $\widetilde\pi:G_S\to\R^2$ mapping $\Gamma_S$ onto the lattice defining $E$. Define
\[
G_R := \ker\left(G_S^m\xrightarrow{\ \Sigma\widetilde\pi\ }\R^2\right).
\]
Its Lie algebra is defined by rational linear equations, so $\Gamma_R:=\Gamma_S^m\cap G_R$ is a cocompact lattice in $G_R$ by Mal'cev's theorem \cite[Chapter~II, Theorem~2.12]{Raghunathan1972}. Thus $R\cong\Gamma_R\backslash G_R$. The elliptic fiber $F$ is induced by the two-dimensional centre $Z_S\subset G_S$. Its diagonal copy $Z_\Delta\subset G_R$ is rational, and therefore
\[
G_N:=G_R/Z_\Delta,
\qquad
\Gamma_N:=\Gamma_R/(\Gamma_R\cap Z_\Delta)
\]
give
\[
N_n\cong\Gamma_N\backslash G_N.
\]
All constructions are preserved by permutation of the factors.

Choose a left-invariant complex coframe $x,\zeta$ on $S$ with $dx=0$ and $d\zeta=x\bar x$. Pulling back to $S^m$ and restricting to $R$ gives $\sum_i x_i=0$. Define
\[
y_i=\sum_{j=1}^m\zeta_j-m\zeta_i.
\]
These forms are basic for the diagonal $F$-action and satisfy $\sum_i y_i=0$. If $b=\sum_jx_j\bar x_j$, then $dy_i=b-mx_i\bar x_i$ and $d\bar y_i=-dy_i$. Set
\[
w_i=\bar x_i,
\qquad
u_i=y_i+\bar y_i,
\qquad
v_i=\frac{y_i-\bar y_i}{2}.
\]
Then $x_i,w_i,u_i$ are closed, $dv_i=b-mx_iw_i$, and each of the four families has sum zero. Thus the complexified invariant coframe of $N_n$ consists of four copies of the standard representation $V$ of $S_m$, and its invariant de Rham algebra is precisely $(E_n,d)$ of \cref{sec:algebra}.

\begin{Prop}\label{prop:geometric-model}
The cdga $\cA_n=(E_n,d)^{S_m}$ is a model of $X_n$ over $\C$. Moreover, $X_n$ is a rational homology manifold of real dimension $4n$, hence a rational Poincar\'e duality space, and
\[
b_2(X_n)=5,
\qquad
b_3(X_n)=0.
\]
\end{Prop}

\begin{proof}
By Nomizu's theorem \cite{Nomizu1954}, the inclusion $E_n\hookrightarrow\Omega^*(N_n;\C)$ is a quasi-isomorphism. It is $S_m$-equivariant, so exactness of finite-group invariants and \cref{lem:finite-quotient} identify $\cA_n$ as a cdga model of $X_n$. Locally, $X_n$ is the quotient of a ball in $\C^{2n}$ by a finite complex-linear group $G$. The link is therefore $S^{4n-1}/G$. Since $G$ preserves the complex orientation, \cref{lem:finite-quotient} applied to the sphere gives
\[
\H^*(S^{4n-1}/G;\Q)
\cong
\H^*(S^{4n-1};\Q)^G
=
\H^*(S^{4n-1};\Q).
\]
Thus every local link is a rational homology sphere, so $X_n$ is a rational homology manifold; equivalently $\IC_{X_n}\cong\Q_{X_n}[2n]$. Finally, \cref{eq:low-cohomology,lem:uniform-H3} give the stated Betti numbers.
\end{proof}

\subsection{The root-lattice cover and Guan's nilmanifold}

The precursor $N_n$ should not be identified with Guan's nilmanifold $M_{n,t}$. They have the same local nilpotent geometry but different natural integral lattices. We now make their relation precise.

Let
\[
A_n=\left\{(a_1,\ldots,a_m)\in\Z^m:\sum_i a_i=0\right\},
\qquad
P_n=\Z^m/\Z(1,\ldots,1).
\]
The first is the root lattice of type $A_n$ and the second is naturally its weight lattice. The map $A_n\to P_n$, $a\mapsto[a]$, is $S_m$-equivariant and has index $m$.

\begin{Prop}\label{prop:guan-precursor-cover}
Assume $d=\deg\cL=mt$. Then Guan's nilmanifold $M_{n,t}$ admits a natural $S_m$-equivariant finite covering
\[
\rho:M_{n,t}\longrightarrow N_n
\]
of degree $m^2=(n+1)^2$. Consequently there is a commutative diagram
\begin{equation}\label{eq:guan-precursor-diagram}
\begin{tikzcd}
M_{n,t} \arrow[r, "\rho"] \arrow[d] & N_n \arrow[d] \\
Y_{n,t}:=M_{n,t}/S_m \arrow[r, "\bar\rho"'] & X_n:=N_n/S_m
\end{tikzcd}
\end{equation}
where $\bar\rho$ is finite of degree $m^2$.
\end{Prop}

\begin{proof}
Let $\Lambda_F$ be the lattice of the elliptic curve $F$, and set $B:=\ker(E^m\xrightarrow{\Sigma}E)$. Over $B$, the fiber of the Bogomolov--Guan precursor is $F^m/F_\Delta$, whose lattice is $P_n\otimes\Lambda_F$. On the other hand, if $K_F:=\ker(F^m\xrightarrow{\Sigma}F)$, then the lattice of $K_F$ is $A_n\otimes\Lambda_F$. The natural inclusion $A_n\hookrightarrow P_n$ therefore induces an isogeny $K_F\to F^m/F_\Delta$. Its kernel is $K_F\cap F_\Delta=F[m]$, so this isogeny has degree $\#F[m]=m^2$.

We now determine when this fiberwise isogeny lifts to the total spaces. The principal elliptic bundle $S\to E$ has Chern class $d=\deg\cL$. Restricting the product bundle over $E^m$ to $B$ and quotienting by $F_\Delta$ gives the $P_n\otimes\Lambda_F$-valued extension class defining $N_n$. Choose the standard root basis $\alpha_i=e_i-e_m$, for $1\le i\le n$, of $A_n$, and the quotient basis $\bar e_i=[e_i]$ of $P_n$. Since $\bar e_m=-\sum_{j=1}^n\bar e_j$, the inclusion $A_n\hookrightarrow P_n$ is represented by
\begin{equation}\label{eq:root-weight-matrix}
G=I_n+\mathbf 1\mathbf 1^{\,t},
\qquad
\det G=m,
\qquad
G^{-1}=I_n-\frac1m\mathbf 1\mathbf 1^{\,t}.
\end{equation}
In the quotient basis, the extension class of the precursor is represented by $dI_n$. Hence a lift to the root lattice is equivalent to finding an integral matrix $C$ satisfying $GC=dI_n$, or equivalently $C=dG^{-1}$. By \eqref{eq:root-weight-matrix}, this matrix is integral precisely when $m\mid d$. Writing $d=mt$, one obtains
\begin{equation}\label{eq:guan-extension-matrix}
C=t\bigl(mI_n-\mathbf 1\mathbf 1^{\,t}\bigr).
\end{equation}
Thus $C$ has diagonal entries $(m-1)t$ and off-diagonal entries $-t$. Up to the choice of root basis and the corresponding sign convention for the central generators, this is exactly the integral central-extension matrix in Guan's definition of $M_{n,t}$; compare \cite{GuanII1995,GuanIII1995}. Therefore the root-lattice lift of the precursor is Guan's nilmanifold $M_{n,t}$. The same computation also makes transparent why the arithmetic condition in Bogomolov's reinterpretation is precisely $m\mid d$; see \cite{Bogomolov1996}.

Finally, the entire construction is $S_m$-equivariant. Since the fiberwise kernel is $F[m]$, the resulting covering $\rho:M_{n,t}\to N_n$ has degree $m^2$. Because the $S_m$-action commutes with $\rho$, the map descends to the quotients and gives the lower horizontal map in \eqref{eq:guan-precursor-diagram}, again of generic degree $m^2$.
\end{proof}

\begin{Rem}\label{rem:same-structure-equations}
The covering in \cref{prop:guan-precursor-cover} explains why the precursor and Guan's nilmanifold have the same structure equations in the de Rham calculation. Passing from $M_{n,t}$ to $N_n$ changes the cocompact lattice inside the simply connected nilpotent Lie group but not its real Lie algebra. Equivalently, the root lattice $A_n\otimes\Lambda_F$ is replaced by the commensurable weight lattice $P_n\otimes\Lambda_F$. Maurer--Cartan equations are Lie-algebraic and therefore do not see this change of lattice. The parameter $t\neq0$ only rescales the nonclosed central directions and can be normalized after extending scalars to $\R$ or $\C$. Hence, $S_m$-equivariantly, the invariant de Rham algebra of $M_{n,t}$ is isomorphic over $\C$ to the same cdga $(E_n,d)$ used for the precursor.

This is exactly the level at which the two constructions agree. Their natural compact nilmanifolds should nevertheless be kept distinct: the integral lattices are different, and the canonical relation between them is the finite covering \eqref{eq:guan-precursor-diagram}, not an identification.
\end{Rem}

\begin{Cor}\label{cor:guan-quotient-model}
For every $t\neq0$, the cdga $\cA_n$ is also a model of $Y_{n,t}=M_{n,t}/S_m$ over $\C$. In particular,
\[
b_2(Y_{n,t})=5,
\qquad
b_3(Y_{n,t})=0.
\]
\end{Cor}

\begin{proof}
By \cref{rem:same-structure-equations}, the $S_m$-equivariant invariant de Rham algebra of $M_{n,t}$ is $(E_n,d)$ after extension of scalars to $\C$. Apply Nomizu's theorem and \cref{lem:finite-quotient}, exactly as in the proof of \cref{prop:geometric-model}.
\end{proof}

\subsection{The map of nonzero degree}

The Hilbert--Chow morphism restricts to a proper map $W\to R/S_m$ and is equivariant for the elliptic action. Since the $F$- and $S_m$-actions commute, it descends to a bimeromorphic map
\[
h:W/F\longrightarrow X_n.
\]
Bogomolov's construction gives a finite morphism $p:Q\to W/F$ of degree $m^2$, see \cite{Bogomolov1996} and \cite[Theorem~3.3]{BKKY2022}. Therefore
\begin{equation}\label{eq:degree-map}
g=h\circ p:Q\longrightarrow X_n,
\qquad
g_*[Q]=m^2[X_n].
\end{equation}
Thus $g$ has nonzero degree. Notice that this map lands in the precursor quotient $X_n$ and this is the map used to transfer nonformality.

\subsection{The third Betti number}

For the computation of the third cohomology group we instead use Guan's quotient $Y_{n,t}$. Guan's original construction produces a smooth compact simply connected holomorphic symplectic manifold $Q_{n,t}$ together with a Hilbert-type resolution
\begin{equation}\label{eq:guan-resolution}
f:Q_{n,t}\longrightarrow Y_{n,t}=M_{n,t}/S_m.
\end{equation}
Bogomolov's construction gives the same family of examples, with the divisibility condition $d=mt$ accounting for the parameter $t$; see \cite{Bogomolov1996} and the comparison in \cite[Theorem~3.3 and Remark~3.4]{BKKY2022}. We therefore for these computations implicitly take the distinguished representative $Q$ associated with the data above to be $Q_{n,t}$ when discussing the resolution.

\begin{Lem}\label{lem:guan-semismall}
The map $f:Q_{n,t}\to Y_{n,t}$ is a projective semismall resolution. If $S_\lambda\subset Y_{n,t}$ is the stratum associated with a partition $\lambda=(\lambda_1,\ldots,\lambda_\ell)$ of $m$, then for a general $s\in S_\lambda$,
\begin{equation}\label{eq:semismall-strata}
\codim_{\C}S_\lambda
=2\bigl(m-\ell(\lambda)\bigr),
\qquad
\dim_{\C}f^{-1}(s)
=m-\ell(\lambda).
\end{equation}
In particular, the unique codimension-two type is $\lambda=(2,1,\ldots,1)$.
\end{Lem}

\begin{proof}
Guan's construction \cite{GuanII1995,GuanIII1995} resolves the local symmetric-product singularities by the corresponding Hilbert--Chow resolutions. The complexified tangent representation of $S_m$ on $M_{n,t}$ is two copies of the standard representation $V$. If the stabilizer of a general point of $S_\lambda$ is $S_{\lambda_1}\times\cdots\times S_{\lambda_\ell}$, then its fixed subspace in $V$ has dimension $\ell-1$. Hence
\[
\dim_\C S_\lambda=2(\ell-1),
\qquad
\codim_\C S_\lambda=2(m-\ell).
\]
Transversely, the resolution is the product of punctual Hilbert--Chow fibers
\[
f^{-1}(s)\cong\prod_{i=1}^{\ell}\operatorname{Hilb}^{\lambda_i}_0(\C^2).
\]
By Brian\c{c}on's theorem \cite{Briancon1977}, the $i$-th factor is irreducible of dimension $\lambda_i-1$. Therefore
\[
\dim_\C f^{-1}(s)=\sum_i(\lambda_i-1)=m-\ell,
\]
which proves semismallness.

We also need projectivity of the morphism, not merely the local Hilbert--Chow description. Let $D\subset Q_{n,t}$ be the reduced exceptional divisor. Since $Q_{n,t}$ is smooth, $D$ is Cartier. For the standard Hilbert--Chow morphism
\[
\rho_r:\operatorname{Hilb}^r(\C^2)\to\operatorname{Sym}^r(\C^2),
\]
Haiman's relative $\operatorname{Proj}$ construction \cite[Proposition~2.6]{Haiman1998} gives a relatively ample line bundle $\cO(1)$, identified with the determinant of the tautological bundle in \cite[Proposition~2.12]{Haiman1998}. If $B_r$ denotes the boundary divisor, then \cite[Lemma~3.7]{Lehn1999} gives
\[
[B_r]=-2c_1\bigl(\det\cO_{\C^2}^{[r]}\bigr),
\]
so with Haiman's normalization $\cO(-B_r)\cong\cO(2)$ and is $\rho_r$-ample. On each local product chart for $f$, the restriction of $D$ is the sum of the pullbacks of the boundary divisors of the Hilbert--Chow factors. Hence $\cO_{Q_{n,t}}(-D)$ is relatively ample on each such chart. Relative ampleness for proper morphisms of complex analytic spaces is local on the base \cite[Definition~2.28]{FujinoMMP}; therefore $\cO_{Q_{n,t}}(-D)$ is $f$-ample and $f$ is projective. In particular, it is cohomologically projective in the sense required by Saito's decomposition theorem \cite[Theorem~(0.3)]{Saito1990}.
\end{proof}

Let $\Sigma_{n,t}\subset Y_{n,t}$ be the codimension-two transposition stratum, and let
\[
\nu:Z_{n,t}\longrightarrow\overline{\Sigma}_{n,t}
\]
be the normalization of its closure. Fix $\tau=(12)\in S_m$ and put $F_\tau=M_{n,t}^\tau$. Let $F_\tau^\circ\subset F_\tau$ be the open subset of points whose stabilizer is exactly $\langle\tau\rangle$, and let $F_\tau^{\mathrm{rel}}$ be the union of the connected components meeting $F_\tau^\circ$.

\begin{Lem}\label{lem:transposition-normalization}
The quotient map $M_{n,t}\to Y_{n,t}$ induces a finite birational morphism
\[
F_\tau^{\mathrm{rel}}/C(\tau)
\longrightarrow
\overline{\Sigma}_{n,t}.
\]
Consequently,
\[
Z_{n,t}\cong F_\tau^{\mathrm{rel}}/C(\tau).
\]
\end{Lem}

\begin{proof}
The fixed locus $F_\tau$ is a smooth complex submanifold and is preserved by the centralizer $C(\tau)$. Since the quotient map is finite, its restriction induces a finite morphism from $F_\tau^{\mathrm{rel}}/C(\tau)$ whose image is $\overline{\Sigma}_{n,t}$. For a general $x\in F_\tau^\circ$, if $\sigma x\in F_\tau^\circ$, then $\sigma^{-1}\tau\sigma$ fixes $x$. Since $\operatorname{Stab}(x)=\langle\tau\rangle$, we get $\sigma^{-1}\tau\sigma=\tau$, hence $\sigma\in C(\tau)$. Thus the map is generically one-to-one. The source is normal because it is a finite quotient of the smooth space $F_\tau^{\mathrm{rel}}$. Therefore the finite birational map is the normalization.
\end{proof}

\begin{Lem}\label{lem:low-degree-sectors}
For $k=2,3$ there is a noncanonical isomorphism
\begin{equation}\label{eq:low-degree-decomposition}
\H^k(Q_{n,t},\C)
\cong
\H^k(Y_{n,t},\C)\oplus \H^{k-2}(Z_{n,t},\C).
\end{equation}
\end{Lem}

\begin{proof}
By \cref{lem:guan-semismall}, $f$ is projective and semismall, so Saito's analytic decomposition theorem \cite[Theorem~(0.3)]{Saito1990} applies. The open stratum contributes $\IC_{Y_{n,t}}\cong\Q_{Y_{n,t}}[2n]$. Indeed, locally $Y_{n,t}$ is a ball modulo a finite complex-linear group; as in the proof of \cref{prop:geometric-model}, its link is a finite orientation-preserving quotient of a sphere and therefore a rational homology sphere. Along $\Sigma_{n,t}$ the generic transverse singularity is of type $A_1$ and the generic fiber is a single exceptional curve $\P^1$, so the top-fiber local system is trivial of rank one.

By \cref{lem:transposition-normalization}, $Z_{n,t}$ is a finite quotient of a smooth complex manifold. The same local-link argument shows that it is a rational homology manifold. Thus $\IC_{Z_{n,t}}\cong\Q_{Z_{n,t}}[2n-2]$. The normalization map $\nu$ is finite and an isomorphism over the dense transposition stratum, hence small, so
\[
\IC_{\overline{\Sigma}_{n,t}}
\cong
\nu_*\Q_{Z_{n,t}}[2n-2].
\]
The remaining relevant strata have codimension at least four. After shifting by $[-2n]$, a support of codimension $2r$ contributes to $\H^k(Q_{n,t})$ through intersection cohomology in degree $k-2r$, which vanishes for $k\le3$ when $r\ge2$. Therefore only the open and transposition strata contribute in degrees two and three, giving \eqref{eq:low-degree-decomposition}.
\end{proof}

\begin{Lem}\label{lem:transposition-connected}
The space $Z_{n,t}$ is connected.
\end{Lem}

\begin{proof}
Taking $k=2$ in \eqref{eq:low-degree-decomposition} gives
\[
b_2(Q_{n,t})=b_2(Y_{n,t})+b_0(Z_{n,t}).
\]
By \cref{cor:guan-quotient-model}, $b_2(Y_{n,t})=5$, while Guan's original computation gives $b_2(Q_{n,t})=6$ for $n\ge2$ \cite{GuanII1995}. Hence $b_0(Z_{n,t})=1$.
\end{proof}

\begin{Lem}\label{lem:kernel-fixed-representation}
Let $\tau=(12)$. If $\sigma\in S_m$ acts trivially on
\[
V^\tau=\{(z_1,\ldots,z_m)\in\C^m:\sum_i z_i=0,\ z_1=z_2\},
\]
then $\sigma\in\langle\tau\rangle$.
\end{Lem}

\begin{proof}
If $m=3$, then $V^\tau$ is the line spanned by $(1,1,-2)$, and the only permutations fixing this vector are the identity and $(12)$. Assume $m\ge4$. For $j,k\ge3$, the vectors $e_j-e_k$ lie in $V^\tau$. A permutation acting trivially on all of these vectors fixes every index $j\ge3$ individually. It therefore preserves the complementary pair $\{1,2\}$, and hence is either the identity or $(12)$.
\end{proof}

\begin{Cor}\label{cor:transposition-identity-component}
Let $F_{\tau,0}$ be the connected component of $F_\tau=M_{n,t}^\tau$ containing the distinguished basepoint $[e]$. Then
\[
F_\tau^{\mathrm{rel}}=F_{\tau,0},
\qquad
Z_{n,t}\cong F_{\tau,0}/C(\tau).
\]
\end{Cor}

\begin{proof}
The $S_m$-action on $M_{n,t}=\Gamma\backslash G$ is induced by automorphisms of $G$ preserving the lattice, so every element of $C(\tau)$ fixes $[e]$ and preserves $F_{\tau,0}$. We claim that $F_{\tau,0}$ meets $F_\tau^\circ$. For $\sigma\notin\langle\tau\rangle$, the fixed locus $F_{\tau,0}\cap M_{n,t}^\sigma$ is a proper analytic subset. Indeed, if it were all of $F_{\tau,0}$, then the differential of $\sigma$ at $[e]$ would act trivially on the tangent space of $F_{\tau,0}$, whose standard-representation part is $V^\tau$, contradicting \cref{lem:kernel-fixed-representation}. Since there are only finitely many $\sigma$, the complement of the union of these proper analytic subsets is nonempty. Every point in that complement has stabilizer exactly $\langle\tau\rangle$.

Thus $F_{\tau,0}$ is one of the components forming $F_\tau^{\mathrm{rel}}$. Because it is $C(\tau)$-stable, its quotient is a connected component of $F_\tau^{\mathrm{rel}}/C(\tau)\cong Z_{n,t}$. By \cref{lem:transposition-connected}, the latter is connected, so no other component occurs.
\end{proof}

\begin{Lem}\label{lem:transposition-H1}
For every $n\geq2$ and every $t\neq0$,
\[
b_1(Z_{n,t})=3.
\]
\end{Lem}

\begin{proof}
By \cref{cor:transposition-identity-component}, $Z_{n,t}\cong F_{\tau,0}/C(\tau)$. The centralizer is $C(\tau)=\langle\tau\rangle\times S_{m-2}$, and $\tau$ acts trivially on $F_{\tau,0}$, so the effective residual action is that of $S_{m-2}$. By \cref{lem:finite-quotient},
\[
\H^1(Z_{n,t},\C)
\cong
\H^1(F_{\tau,0},\C)^{S_{m-2}}.
\]

The identity component $F_{\tau,0}$ is a compact nilmanifold: its fixed Lie subgroup is defined by rational equations and its intersection with the lattice is cocompact. By Nomizu's theorem, its first cohomology is computed by invariant forms. The invariant coframe is obtained by restricting the four copies of $V$ in \cref{rem:same-structure-equations} to $V^\tau$. As an $S_{m-2}$-representation,
\begin{equation}\label{eq:fixed-representation}
V^\tau\cong\C\oplus V_{m-2},
\end{equation}
where $V_{m-2}$ is the standard representation of $S_{m-2}$; for $m=3$ the second summand is zero. The trivial summand is generated by
\[
e_0=(m-2,m-2,-2,\ldots,-2).
\]
Thus each family $x,w,u,v$ contributes one invariant degree-one form $x_0,w_0,u_0,v_0$. The first three are closed. In the normalized structure equations,
\[
dv_0=-m\sum_i(e_0)_i x_iw_i.
\]
This two-form is nonzero on $V^\tau$. For $m=3$, evaluate the corresponding bilinear form on the $x$- and $w$-vectors both equal to $(1,1,-2)$; for $m\ge4$, use $(0,0,1,-1,0,\ldots,0)$ in both factors. Hence $v_0$ is not closed. There are no degree-one boundaries, so the invariant closed degree-one forms are exactly $x_0,w_0,u_0$. Therefore $\H^1(Z_{n,t},\C)\cong\C^3$.
\end{proof}

\begin{Prop}\label{prop:BG-b3}
Let $X$ be a $\BG_n$-manifold with $n\geq2$. Then
\begin{equation}\label{eq:BG-b3}
b_3(X)=3.
\end{equation}
In particular, by Poincar\'e duality, $b_{4n-3}(X)=3$.
\end{Prop}

\begin{proof}
For every distinguished representative $Q_{n,t}$, \eqref{eq:low-degree-decomposition} gives
\[
\H^3(Q_{n,t},\C)
\cong
\H^3(Y_{n,t},\C)\oplus \H^1(Z_{n,t},\C).
\]
By \cref{cor:guan-quotient-model}, $\H^3(Y_{n,t},\C)=0$, while \cref{lem:transposition-H1} gives $\H^1(Z_{n,t},\C)\cong\C^3$. Thus $b_3(Q_{n,t})=3$ for every $t\neq0$. Betti numbers are constant in smooth proper families by Ehresmann's fibration theorem, so the same holds for every manifold deformation equivalent to one of these representatives.
\end{proof}

\begin{Rem}\label{rem:b3-kahler}
The equality $b_3=3$ already gives a topological obstruction to K\"ahlerness: every compact K\"ahler manifold has even odd Betti numbers. The same parity conclusion holds for compact complex manifolds satisfying the $\partial\bar\partial$-lemma. Thus the underlying smooth manifold of a $\BG_n$-manifold admits neither a K\"ahler complex structure nor a complex structure satisfying the $\partial\bar\partial$-lemma. The nonformality result below is stronger, since it gives an obstruction at the level of the rational homotopy type.
\end{Rem}

\subsection{Proof of the main theorem}

\begin{proof}[Proof of \cref{thm:main}]
By \cref{prop:geometric-model,thm:uniform-algebra}, the quotient $X_n$ is nonformal. The map $g:Q\to X_n$ of \eqref{eq:degree-map} has nonzero degree, and $X_n$ is a rational Poincar\'e duality space. Hence the domination theorem of Milivojevi\'c--Stelzig--Zoller \cite[Theorem~A]{MSZ2023} implies that $Q$ cannot be rationally formal: otherwise $X_n$ would be formal, contradicting
\cref{thm:uniform-algebra}.

Thus $Q$ is nonformal over $\Q$. Formality descends along characteristic-zero field extensions \cite[Remark~3.7]{MSZ2023}, so $Q$ is nonformal over $\R$ and $\C$ as well. Finally, a smooth proper family is differentiably locally trivial by Ehresmann's theorem. Formality is a homotopy invariant, so every manifold deformation equivalent to $Q$ is nonformal. Since the distinguished representative $Q=Q_{n,t}$ was arbitrary, every $\BG_n$-manifold is nonformal.
\end{proof}

\begin{Cor}\label{cor:domination}
Let $X$ be a $\BG_n$-manifold with $n\ge2$. No formal closed oriented manifold of real dimension $4n$ admits a map of nonzero degree to $X$. In particular, no compact K\"ahler manifold of real dimension $4n$ dominates $X$. Moreover, the underlying smooth manifold of $X$ admits no complex structure satisfying the $\partial\bar\partial$-lemma.
\end{Cor}
\begin{proof}
By \cref{thm:main}, $X$ is nonformal, while every closed oriented manifold is a rational Poincar\'e duality space. The first statement is therefore the contrapositive of the domination theorem \cite[Theorem~A]{MSZ2023}. Compact K\"ahler manifolds, and more generally compact complex manifolds satisfying the $\partial\bar\partial$-lemma, are formal \cite{DGMS1975}. The last assertion also follows independently from \cref{prop:BG-b3}.
\end{proof}

\section{Further directions}
\label{sec:questions}

The results above concern the de Rham homotopy type and the first nontrivial odd Betti number of Bogomolov--Guan manifolds. A more detailed study of their Hodge-theoretic invariants will be carried out in the companion paper \cite{RiosOrtizBGTopological}. Specifically, for the Bogomolov--Guan fourfold we compute the complete Hodge diamond, prove degeneration of the Fr\"olicher spectral sequence at the first page, and describe the contributions of the different semismall strata to Dolbeault cohomology.

\begin{Ques}
Does the Fr\"olicher spectral sequence degenerate at $E_1$ for every complex manifold in each Bogomolov--Guan deformation class of type $\BG_n$? If not, can one describe the locus in the Bogomolov--Guan deformation space where the degeneration page jumps?
\end{Ques}

This is a genuinely deformation-theoretic question. In contrast with the de Rham nonformality and Betti-number computations of the present paper, the Fr\"olicher spectral sequence depends on the complex structure and is not determined by the underlying smooth manifold. Recall that $E_1$-degeneration is stable under sufficiently small deformations: this follows from upper semicontinuity of the Dolbeault numbers together with constancy of the Betti numbers in a smooth proper family. Thus the distinguished fourfold lies in an open subset of the Bogomolov--Guan deformation space on which the Fr\"olicher spectral sequence degenerates at $E_1$. The natural problem is whether this open locus is the entire deformation space. If it is not, one may ask how its complement is situated in the period space and which geometric features of the complex structure are responsible for a jump of the degeneration page.

The four-dimensional calculation also suggests that the semismall decomposition of $f:Q_{n,t}\to Y_{n,t}$ should provide an effective way to separate the contribution of the open stratum from those of the higher collision strata. In higher dimension, the main additional difficulty is to understand the Dolbeault and Hodge-theoretic information carried by the normalizations of these strata and the corresponding local systems appearing in the decomposition theorem.

\begin{Ques}
Are Bogomolov--Guan manifolds Dolbeault formal? If not, can their nonformality be detected by explicit Dolbeault, Bott--Chern, or Aeppli Massey-type operations?
\end{Ques}

This question is independent of the de Rham nonformality established here. The techniques developed by Sferruzza--Tomassini \cite{SferruzzaTomassini2022} and Sferruzza--Verbitsky \cite{SV2026} suggest that the invariant complex geometry of the Bogomolov--Guan precursor may provide a natural starting point. A central problem is to understand whether suitable complex-homotopical obstructions can first be constructed on the invariant model and then controlled through the finite quotient and the subsequent resolution.

\begin{Ques}
To what extent is the rational or complex homotopy type of a
$\BG_n$-manifold determined by the two commensurable symmetric quotients
$X_n=N_n/S_{n+1}$ and $Y_{n,t}=M_{n,t}/S_{n+1}$?
\end{Ques}

The nonzero-degree map $Q\to X_n$ is sufficient to transfer nonformality, whereas the degree-one semismall resolution $Q_{n,t}\to Y_{n,t}$ recovers $\H^2(Q_{n,t})$ and $\H^3(Q_{n,t})$ from the open and transposition strata. Since $X_n$ and $Y_{n,t}$ have the same invariant de Rham model but different integral lattices, it would be interesting to determine which higher homotopy operations are already visible in $\cA_n$ and which require the additional topology and local systems of the collision strata of $Y_{n,t}$.

\bibliography{references}

\begin{thebibliography}{DGMS75}

\bibitem[AKV26]{AKV2026}
Leila Abubakarova, Alexandra Kuznetsova, and Misha Verbitsky.
\newblock Projective subvarieties of {B}ogomolov--{G}uan manifolds and
  quasi-diagonals in products of elliptic curves, 2026.
\newblock v1.

\bibitem[BG76]{BousfieldGugenheim1976}
A.~K. Bousfield and V.~K. A.~M. Gugenheim.
\newblock {\em On PL de Rham theory and rational homotopy type}.
\newblock Number 179 in Memoirs of the American Mathematical Society. American
  Mathematical Society, 1976.

\bibitem[BKKY22]{BKKY2022}
Fedor Bogomolov, Nikon Kurnosov, Alexandra Kuznetsova, and Egor Yasinsky.
\newblock Geometry and automorphisms of non-{K}\"ahler holomorphic symplectic
  manifolds.
\newblock {\em Int. Math. Res. Not. IMRN}, 2022(16):12302--12341, 2022.

\bibitem[Bog96]{Bogomolov1996}
Fedor~A. Bogomolov.
\newblock On {G}uan's examples of simply connected non-{K}\"ahler compact
  complex manifolds.
\newblock {\em Amer. J. Math.}, 118(5):1037--1046, 1996.

\bibitem[Bri77]{Briancon1977}
Jo\"el Brian\c{c}on.
\newblock Description de $\operatorname{Hilb}^n\mathbb{C}\{x,y\}$.
\newblock {\em Inventiones Mathematicae}, 41:45--89, 1977.

\bibitem[CFM08]{CFM2008}
Gil~R. Cavalcanti, Marisa Fern{\'a}ndez, and Vicente Mu{\~n}oz.
\newblock Symplectic resolutions, {L}efschetz property and formality.
\newblock {\em Adv. Math.}, 218(2):576--599, 2008.

\bibitem[DGMS75]{DGMS1975}
Pierre Deligne, Phillip Griffiths, John Morgan, and Dennis Sullivan.
\newblock Real homotopy theory of {K}\"ahler manifolds.
\newblock {\em Invent. Math.}, 29(3):245--274, 1975.

\bibitem[FM08]{FernandezMunoz2008}
Marisa Fern{\'a}ndez and Vicente Mu{\~n}oz.
\newblock An 8-dimensional nonformal, simply connected, symplectic manifold.
\newblock {\em Ann. of Math. (2)}, 167(3):1045--1054, 2008.

\bibitem[FOT08]{FelixOpreaTanre2008}
Yves F{\'e}lix, John Oprea, and Daniel Tanr{\'e}.
\newblock {\em Algebraic Models in Geometry}, volume~17 of {\em Oxford Graduate
  Texts in Mathematics}.
\newblock Oxford University Press, Oxford, 2008.

\bibitem[Fuj22]{FujinoMMP}
Osamu Fujino.
\newblock Minimal model program for projective morphisms between complex
  analytic spaces, 2022.

\bibitem[Gua95a]{GuanII1995}
Daniel Guan.
\newblock Examples of compact holomorphic symplectic manifolds which are not
  {K}\"ahlerian. {II}.
\newblock {\em Invent. Math.}, 121(1):135--145, 1995.

\bibitem[Gua95b]{GuanIII1995}
Daniel Guan.
\newblock Examples of compact holomorphic symplectic manifolds which are not
  {K}\"ahlerian. {III}.
\newblock {\em Internat. J. Math.}, 6(5):709--718, 1995.

\bibitem[Gua25]{Guan2025}
Daniel Guan.
\newblock Examples of compact simply connected holomorphic symplectic manifolds
  which are not formal.
\newblock {\em Axioms}, 14(3), 2025.

\bibitem[Hai98]{Haiman1998}
Mark~D. Haiman.
\newblock $t,q$-{Catalan} numbers and the {Hilbert} scheme.
\newblock {\em Discrete Mathematics}, 193(1--3):201--224, 1998.

\bibitem[Has05]{Hasegawa2005}
Keizo Hasegawa.
\newblock Complex and {K\"ahler} structures on compact solvmanifolds.
\newblock {\em Journal of Symplectic Geometry}, 3(4):749--767, 2005.

\bibitem[Ill78]{Illman1978}
S{\"o}ren Illman.
\newblock Smooth equivariant triangulations of {$G$}-manifolds for {$G$} a
  finite group.
\newblock {\em Math. Ann.}, 233(3):199--220, 1978.

\bibitem[KV19]{KurnosovVerbitsky2019}
Nikon Kurnosov and Misha Verbitsky.
\newblock Deformations and {BBF} form on non-{K}\"ahler holomorphically
  symplectic manifolds, 2019.

\bibitem[Leh99]{Lehn1999}
Manfred Lehn.
\newblock Chern classes of tautological sheaves on {Hilbert} schemes of points
  on surfaces.
\newblock {\em Inventiones Mathematicae}, 136(1):157--207, 1999.

\bibitem[Mil79]{Miller1979}
Timothy~J. Miller.
\newblock On the formality of $(k-1)$-connected compact manifolds of dimension
  less than or equal to $4k-2$.
\newblock {\em Illinois J. Math.}, 23(2):253--258, 1979.

\bibitem[MSZ23]{MSZ2023}
Aleksandar Milivojevi{\'c}, Jonas Stelzig, and Leopold Zoller.
\newblock Formality is preserved under domination, 2023.

\bibitem[Nom54]{Nomizu1954}
Katsumi Nomizu.
\newblock On the cohomology of compact homogeneous spaces of nilpotent {L}ie
  groups.
\newblock {\em Ann. of Math. (2)}, 59(3):531--538, 1954.

\bibitem[Rag72]{Raghunathan1972}
M.~S. Raghunathan.
\newblock {\em Discrete Subgroups of {L}ie Groups}, volume~68 of {\em
  Ergebnisse der Mathematik und ihrer Grenzgebiete}.
\newblock Springer-Verlag, Berlin--Heidelberg, 1972.

\bibitem[RO26]{RiosOrtizBGTopological}
{\'A}ngel~David R{\'i}os~Ortiz.
\newblock Topological invariants of bogomolov--guan fourfolds.
\newblock In preparation, 2026.

\bibitem[Sai90]{Saito1990}
Morihiko Saito.
\newblock Decomposition theorem for proper {K{\"a}hler} morphisms.
\newblock {\em Tohoku Mathematical Journal, Second Series}, 42(2):127--148,
  1990.

\bibitem[ST22]{SferruzzaTomassini2022}
Tommaso Sferruzza and Adriano Tomassini.
\newblock Dolbeault and {B}ott--{C}hern formalities: deformations and
  {$\partial\overline{\partial}$}-lemma.
\newblock {\em J. Geom. Phys.}, 175:Paper No. 104470, 19, 2022.

\bibitem[Sul77]{Sullivan1977}
Dennis Sullivan.
\newblock Infinitesimal computations in topology.
\newblock {\em Inst. Hautes {\'E}tudes Sci. Publ. Math.}, 47:269--331, 1977.

\bibitem[SV26]{SV2026}
Tommaso Sferruzza and Misha Verbitsky.
\newblock Dolbeault formality for complex nilmanifolds, 2026.

\end{thebibliography}
\bibliographystyle{alpha}

\end{document}